\documentclass[11pt]{article}
\usepackage{amsmath,amsthm,amssymb}
\usepackage{color,textcomp}
\usepackage[utf8]{inputenc}

\usepackage[hyperfootnotes=false]{hyperref}

\usepackage{pxfonts}
\usepackage[T1]{fontenc}
\usepackage{titlesec}
\titleformat{\paragraph}
{\normalfont\normalsize\bfseries}{\theparagraph}{1em}{}
\titlespacing*{\paragraph}
{0pt}{3.25ex plus 1ex minus .2ex}{1.5ex plus .2ex}

\def\titlerunning#1{\gdef\titrun{#1}}
\makeatletter
\def\author#1{\gdef\autrun{\def\and{\unskip, }#1}\gdef\@author{#1}}
\def\address#1{{\def\and{\\\hspace*{18pt}}\renewcommand{\thefootnote}{}%
\footnote {#1}}%
\markboth{\autrun}{\titrun}}
\makeatother
\def\email#1{\hspace*{4pt}{\em e-mail}: #1}
\def\MSC#1{{\renewcommand{\thefootnote}{}%
\footnote{\emph{Mathematics Subject Classification (2020):} #1}}}
\def\keywords#1{\par\medskip
\noindent\textbf{Keywords:} #1}

\newtheorem{theorem}{Theorem}[section]
\newtheorem{prop}[theorem]{Proposition}
\newtheorem{cor}[theorem]{Corollary}
\newtheorem{lemma}[theorem]{Lemma}

\newtheorem{conj}[theorem]{Conjecture}

\theoremstyle{definition}

\newtheorem{remark}[theorem]{Remark}

\numberwithin{equation}{section}

\def\0{\mathbf 0}

\def\F{{\mathbb F}}

\begin{document}


\baselineskip=16pt


\titlerunning{}

\title{Zeros of special polynomials and their impact on a class of APN functions}

\author{Daniele Bartoli \and Marco Calderini \and Giuseppe Marino \and Francesco Pavese}

\date{}

\maketitle

\address{D. Bartoli: Department of Mathematics and Informatics, University of Perugia, Perugia, Italy; \email{daniele.bartoli@unipg.it}
\and
M. Calderini: Department of Mathematics, University of Trento, Trento, Italy; \email{marco.calderini@unitn.it}
\and
G. Marino: Department of Mathematics and Applications ``R. Caccioppoli'', University of Naples ``Federico II'', Naples, Italy; \email{giuseppe.marino@unina.it}
\and
F. Pavese: Department of Mechanics, Mathematics and Management, Polytechnic University of Bari, Bari, Italy; \email{francesco.pavese@poliba.it}
}

\bigskip

\MSC{Primary 51E20; Secondary 05B25}


\begin{abstract}
In 2021, Calderini et al. introduced a construction for APN functions on $\mathbb{F}_{2^{2m}}$ in bivariate form 
\begin{align*}
& f(x,y)=\big(xy,\, x^{2^r+1} + x^{2^{r+m/2}} y^{2^{m/2}} + bxy^{2^r} + cy^{2^r+1}\big) & r < m/2, \gcd(r, m) = 1. 
\end{align*}
They showed that this family exists provided the existence of a polynomial 
\begin{align*}
P_{c,b}(X)=(cX^{2^r +1} + b X^{2^r}+1)^{2^{m/2}+1}+X^{2^{m/2}+1}, 
\end{align*}
with no zeros in $\mathbb{F}_{2^{2m}}$.
For $m\le 6$ it was shown that we can have APN functions belonging to this family. However, up to now, no construction of such polynomials is known for $m\ge 8$. In this work we provide a non-existence result of such functions whenever $r<m/8-1$, by application of techniques from algebraic varieties over finite fields.
In particular, for $r=1$ we have that the construction of Calderini et al. cannot provide an APN function for $m\ge 8$.

\end{abstract}

\section{Introduction}

Let $\mathbb{F}_{2^n}$ be the finite field with $2^n$ elements. A function $f$ from $\mathbb{F}_{2^n}$ into itself is called almost perfect nonlinear (APN) if for any non zero $a\in \mathbb{F}_{2^n}$ and any $b\in \mathbb{F}_{2^n}$, the equation 
$$
f(x+a)+f(x)=b
$$
admits at most two solutions.

APN functions play a central role in modern cryptography since they provide optimal resistance against differential cryptanalysis (\cite{diffBS}) when used as substitution boxes in block ciphers. Beyond cryptography, they also appear as optimal objects in coding theory, combinatorics, and projective geometry \cite{semibi,DO,hypDE}. 

Despite their importance, only a few infinite families of APN functions are currently known, and their classification up to CCZ- or EA-equivalence remains an open problem (see \cite{LK} for a list of known APN families and for the definition of these equivalence relations).



Several of the known families that can be defined in even dimension have been obtained using the so called bivariate construction introduced by Carlet in \cite{carletbiv}. In particular, let $n=2m$, we can decompose $\mathbb{F}_{2^n}$ as $\mathbb{F}_{2^m}\times\mathbb{F}_{2^m}$, and a function from $\mathbb{F}_{2^m}\times\mathbb{F}_{2^m}$ into itself can be represented using bivariate polynomials, that is, $f(x,y)=(f_1(x,y),f_2(x,y))$ with $f_1,f_2\in\mathbb{F}_{2^m}[x,y]$.

In \cite{carletbiv}, the author considered functions $f(x,y)$ where $f_1(x,y)$ was given by the Maiorana-McFarland function $xy$, and provided some necessary and sufficient conditions for the APN property of $f(x,y)$. He also introduced a class of APN function in bivariate form which was later proved (see \cite{APN}) to be equivalent to the hexanomial family constructed in \cite{BChex}.
The bivariate construction was later used for obtaining other classes of APN functions (\cite{APN,tan,ZP}). Recently, in \cite{Golbiproj}, G\"olo\u{g}lu
proposed a generalization of the bivariate construction based on the so-called biprojective polynomials. Bi-projective polynomials has been used for constructing several classes of APN functions lately \cite{CLVbiproj,GKbiproj,LZLQ}.

Within specific families, the APN property is intrinsically connected to the existence of polynomials with well-defined structural properties. Accordingly, a fundamental problem is to determine whether APN functions derived from these constructions exist in infinitely many dimensions or whether they are restricted to finitely many instances \cite{BCPZ,Bhex,BTThex,Golhex}.

In particular, the existence of several classes of bivariate APN families constructed to date relies on the fact that a certain projective polynomial, that is a polynomial of type $x^{2^r+1}+x+a$, admits no roots over $\mathbb{F}_{2^m}$ \cite{CLVbiproj,carletbiv,tan}.

Projective polynomials and their roots have been studied in several works, such as \cite{Bluher,BTThex,HKproj1,HKproj2}.
So, applying Bluher’s results (\cite{Bluher}), one obtains that these constructions yield an APN function in every dimension in which they are defined.

 For the case of the APN class introduced in \cite{APN}, the existence of instances coming from this construction is related to the roots of a certain polynomial, which is not projective.

In particular, the APN class given in \cite{APN} is the following:
\begin{theorem}\label{Th:APN}\cite[Theorem 6.2]{APN}
Let $n = 2m$ with $m$ even, and let $r < m/2$ be such that $\gcd(r,m) = 1$.  Consider $b,c \in \mathbb{F}_{2^{m}}$
Then
\[
f_{b,c,r}(x,y) = \big(xy,\, x^{2^r+1} + x^{2^{r+m/2}} y^{2^{m/2}} + bxy^{2^r} + cy^{2^r+1}\big)
\]
is APN if and only if
\[
P_{c,b}(X)=\big(cX^{2^r+1} + bX^{2^r} + 1\big)^{2^{m/2}+1} + X^{2^{m/2}+1}
\]
has no zero in $\mathbb{F}_{2^m}$.
\end{theorem}

The authors showed that for $n\le 12$ (so $m\le 6$) it was possible to produce new APN functions (up to CCZ-equivalence). However, if such functions exist also for higher dimensions is an open problem.

The aim of this work is to investigate such an open question. In particular, we prove that for each $r <m/8-1$ there are no instances of $b,c \in \mathbb{F}_{2^{m}}$ for which $f_{b,c,r}(x,y)$ is APN.

The main tool is given by application of techniques from algebraic varieties over finite fields. 

Denote $q=2^{m/2}$. 
First observe that $P_{c,b}(X)$ has a zero in $\mathbb{F}_{q^2}$ if and only if there exists $x\in \mathbb{F}_{q^2}^*$ such that 
\begin{align}
\frac{cx^{2^r+1} + bx^{2^r} + 1}{x}  \label{eq0}
\end{align}
is a $(q+1)$-root of unity. This is equivalent to ask that 
\begin{align*}
\frac{cx^{2^r+1} + bx^{2^r} + 1}{x}=\frac{x^q}{c^qx^{(2^r+1)q} + b^qx^{2^rq} + 1}.
\end{align*}
Let $x=x_0+\xi x_1$, where $\{1,\xi\}$ is an $\mathbb{F}_q$ basis of $\mathbb{F}_{q^2}$ and $x_0,x_1 \in \mathbb{F}_q$. The previous condition (since $x\neq 0$) can be equivalently rewritten as 
\begin{align*}
\left(c(x_0+\xi x_1)^{2^r+1} + b(x_0+\xi x_1)^{2^r} + 1\right)\left(c^q(x_0+\xi^q x_1)^{2^r+1} + b^q(x_0+\xi^q x_1)^{2^r} + 1\right)+(x_0+\xi x_1)(x_0+\xi^q x_1)=0.
\end{align*}
In order to prove that for each $b,c \in \mathbb{F}_{q^2}$ there is at least a solution $(\overline{x_0},\overline{x_1}) \in \mathbb{F}_q^2$ to the above equation, we consider the algebraic curve $\mathcal{D}_{b,c,r}$ defined by
\begin{align*}
\left(c(X+\xi Y)^{2^r+1} + b(X+\xi Y)^{2^r} + 1\right)\left(c^q(X+\xi^q Y)^{2^r+1} + b^q(X+\xi^q Y)^{2^r} + 1\right)+(X+\xi Y)(X+\xi^q Y)=0.
\end{align*}
Via the change of variables $(X+\xi Y,X+\xi^q Y)\mapsto (X,Y)$, $\mathcal{D}_{b,c,r}$ is affinely equivalent to the plane curve 
$\mathcal{C}_{b,c,r}$ defined by
\begin{align*}
\left(cX^{2^r+1} + bX^{2^r} + 1\right)\left(c^qY^{2^r+1} + b^qY^{2^r} + 1\right)+XY=0.
\end{align*}
Our strategy consists in proving that $\mathcal{C}_{b,c,r}$, $b, c \in \F_{q^2}$, $c \ne 0$, $r \ge 1$, is absolutely irreducible and so is $\mathcal{D}_{b,c,r}$. Hence, by the Hasse-Weil bound we obtain the existence of at least one point $(\overline{x_0},\overline{x_1}) \in \mathbb{F}_q^2$ in $\mathcal{D}_{b,c,r}$. The case $c = 0$ is treated separataly. Therefore, by Theorem \ref{Th:APN} the function $f_{b,c,r}(x,y)$ is not APN.

\section{Preliminary results}

We now recall some basic facts on curves/surfaces over (finite) fields. For more details, we refer  to~\cite{fulton,HKT}, or the reader's favorite algebraic geometry book. As customary, for a field $\F$, we denote by $\overline{\F}$ its algebraic closure, and by $\mathbb{P}^m(\F)$ (respectively, $\mathbb{A}^m(\F)$)  the $m$-dimensional projective (respectively, affine) space over the field $\F$. Solutions of one or more polynomial equations form what we call algebraic hypersurfaces or varieties. An algebraic hypersurface defined over a field $\mathbb{F}$ is called absolutely irreducible if the associated polynomial is irreducible over every algebraic extension of $\mathbb{F}$. An absolutely irreducible $\mathbb{F}$-rational component of a hypersurface defined by a polynomial $F$ is  an absolutely irreducible hypersurface, associated to a factor of $F$ defined over $\overline{\F}$.

In two dimensions, $\mathcal{C}$ is an affine curve over a field $\F$ if it is the zero set of a polynomial $F(X,Y)\in\F[X,Y]$. A projective curve $\mathcal{C}$  over a field $\F$ is the zero set of a homogeneous polynomial $F(X,Y,Z)\in\F[X,Y,Z]$.  The polynomial $F$ is the defining polynomial of $\mathcal{C}$.

Finally, we will make use of the following version of the celebrated Hasse-Weil theorem. 
\begin{theorem}[Aubry-Perret bound \textup{\cite[Corollary 2.5]{AubryPerret}}]
\label{Th:AubryPerret}
Let $\mathcal{C}\subset \mathbb{P}^{n}(\mathbb{F}_q)$ be an absolutely irreducible curve which is a complete intersection of $(n-1)$ hypersurfaces of degrees $d_1,\ldots, d_{n-1}$ and set $d=\prod_{i=1}^{n-1}d_i$. Then the number $\mathcal{C}(\mathbb{F}_q)$ of $\mathbb{F}_q$-rational points of $\mathcal{C}$ in $ \mathbb{P}^{n}(\mathbb{F}_q)$ satisfies

\vspace*{-0.5 cm}
\begin{equation}\label{EQ:HW3}
q+1-(d-1)(d-2)\sqrt{q}\leq \#\mathcal{C}(\mathbb{F}_q)\leq q+1+(d-1)(d-2)\sqrt{q}.
\end{equation}
\end{theorem}

\section{The case $c=0$}

In this section we will consider the case $c=0$ and prove that for any $b\in\mathbb{F}_{q^2}^*$
the polynomial $P_{0,b}(X)=\big(bX^{2^r} + 1\big)^{q+1} + X^{q+1}$ has a zero in $\mathbb{F}_{q^2}$, and thus,
we cannot have an APN function of the form $f_{b, 0, r}$, where $f_{b,c,r}$ is as in Theorem \ref{Th:APN}.

For this purpose let us denote $\mu_{q+1}:=\{x^{q-1}\,:\,x\in \mathbb{F}_{q^2}^*\}$ the set of $(q+1)$th root of unity in $\mathbb{F}_{q^2}$. Any element $b$ of $\mathbb{F}_{q^2}^*$ can be uniquely represented
as $b=ut$ for some $t\in\mathbb{F}_{q}^*$ and $u \in \mu_{q+1}$.

Let us recall the following well-known result, which comes from the Hilbert’s Theorem 90.

\begin{lemma}\label{lm:H90}
Let $\alpha\in\mathbb{F}_{2^m}$ and let $j$ be such that $\gcd(j,m)=1$. Then, $Tr_{2}^{2^m}(\alpha)=0$ if and only if there exists $\beta\in\mathbb{F}_{2^m}$ such that $\alpha=\beta^{2^j}-\beta$. 
Here $Tr_{2}^{2^m}$ is the trace map from $\mathbb{F}_{2^m}$ onto $\mathbb{F}_2$.
\end{lemma}

\begin{lemma}
Let $b\in \mathbb{F}_{q^2}^*$, and $r$ be such that $\gcd(r,m)=1$. Then, the polynomial
    $\big(bX^{2^r} + 1\big)^{q+1} + X^{q+1}$ has a zero in $\mathbb{F}_{q^2}$.
\end{lemma}
\begin{proof}
    We note that $\big(bX^{2^r} + 1\big)^{q+1} + X^{q+1}$ has a zero in $\mathbb{F}_{q^2}$ if and only if there exist $x\in\mathbb{F}_{q^2}$ and $u\in\mu_{q+1}$ such that 
    \begin{equation}\label{eq:c0}
    bx^{2^r} + ux+1=0.
    \end{equation}
    Now, $b=u't$ for some $t\in\mathbb{F}_{q}^*$ and $u' \in \mu_{q+1}$. Therefore, performing the substitution $x\mapsto b^{-2^{-r}}x$ and considering $u=u'^{2^{-r}}\in\mu_{q+1}$, Equation \eqref{eq:c0} becomes
    \begin{equation}\label{eq:c02}
        x^{2^r}+t'x+1=0,
    \end{equation}
    where $t'=t^{-2^{-r}}$.

    Let us note that $\gcd(2^r-1,q-1)=1$, so there exists $\bar{t}\in\mathbb{F}_q^*$ such that $\bar{t}^{2^r-1}=t'$. Therefore, substituting $x\mapsto \bar tx $ in Equation \eqref{eq:c02} and dividing by $\bar t ^{2^r}$ we obtain 
    $$
    x^{2^r}+x=\bar{t}^{-2^r}.
    $$
    Now, since $\bar{t}^{-2^r}$ is an element of $\mathbb{F}_q$ we have that $Tr_{2}^{q^2}\left(\bar{t}^{-2^r}\right)=0$. Hence, being $\gcd(r,m)=1$, from Lemma \ref{lm:H90} we have that there exists an element in $\mathbb{F}_{q^2}$ such that $x^{2^r}+x=\bar{t}^{-2^r}$.
\end{proof}

As a consequences we get the following:
\begin{theorem}\label{th:b0}
Let $n = 2m$ with $m$ even, and let $r < m/2$ be such that $\gcd(r,m) = 1$.  Then, for any $b\in \mathbb{F}_{2^{m}}^*$, the function $f_{b,0,r}(x,y)$ defined as in Theorem \ref{Th:APN} is not APN.
\end{theorem}

\section{On the irriducibility of $\mathcal{C}_{b,c,r}$}

\begin{theorem}\label{th:Cbc}
    The curve $\mathcal{C}_{b,c,r} : F_{b,c,r}(X,Y)=0$, where 
    \begin{align*}
    	F_{b,c,r}(X,Y)=\left(c X^{2^r+1}+bX^{2^r}+1\right)\left(c^q Y^{2^r+1}+b^qY^{2^r}+1\right)-XY,
    \end{align*}
    $b,c \in \mathbb{F}_{q^2}$, $c\neq 0$, $r\geq1$, is absolutely irreducible for any $b,c\in \mathbb{F}_{q^2}$.
\end{theorem}
\begin{proof}

    First observe that there is no non-constant factor $G(X,Y)\in \overline{\mathbb{F}_{q}}[X,Y]$ of $F_{b,c,r}$ of degree $0$ in $X$ or $Y$. By way of contradiction suppose that $G(X) \mid F_{b,c,r}(X,Y)$: then $G(X)\mid GCD((c X^{2^r+1}+bX^{2^r}+1),X)=1$, a contradiction to $G$ being non-trivial.

    Consider $F_{b,c,r}(X,Y)=G(X,Y)H(X,Y)$, $G,H\in \overline{\mathbb{F}_{q}}[X,Y]$. To be more explicit,
    \begin{eqnarray*}
       G(X,Y) &:=& g_0(X)Y^s+g_1(X)Y^{s-1}+\cdots + g_s(X);\\
       H(X,Y) &:=& h_0(X)Y^{2^r+1-s}+h_1(X)Y^{2^r+1-s-1}+\cdots + h_{2^r+1-s}(X),
    \end{eqnarray*}
    where each $g_i$ and $h_i$ are polynomials in $X$. Without loss of generality, we can suppose that $1\leq s\leq 2^{r-1}$. Also 
    \begin{align*}
        g_0(X)h_0(X)=c^q(c X^{2^r+1}+bX^{2^r}+1)
    \end{align*}
    and thus, since $c\neq 0$, $GCD(g_0(X),h_0(X))=1$. Comparing the coefficient of $Y^{2^r}$ in $G(X,Y)H(X,Y)$ and in $F_{b,c,r}(X,Y)$ we deduce
    $$g_0(X)h_1(X)+g_1(X)h_0(X)=b^q (c X^{2^r+1}+bX^{2^r}+1).$$
    \begin{enumerate}
        \item In the case where $\deg(g_0(X))>0$, from the equations above we deduce that $g_0(X)\mid g_1(X)h_0(X)$ and thus $g_0(X)\mid g_1(X)$ since $g_0$ and $h_0$ are coprime.

        Consider now the coefficient of $Y^{2^r+1-\ell}$, $\ell\in [2 \ldots s]$, in  $G(X,Y)H(X,Y)$ and in $F_{b,c,r}(X,Y)$. Since they must coincide and $2^r+1-s\geq 2^r+1-2^{r-1}=2^{r-1}+1>1$, we have that
        $$\sum_{i=0}^{\ell}g_i(X)h_{\ell-i}(Y) =0.$$
        Thus, using induction, $g_0(X)\mid g_\ell(X)$, $\ell=2,\ldots, s$. Therefore $g_0(X)\mid G(X,Y)$ and we find a non-costant factor of $F_{b,c,r}$ of degree $0$ in $Y$, a contradiction. 

        \item In the case where $\deg(g_0(X))=0$, then $h_0(X)=\lambda (c X^{2^r+1}+bX^{2^r}+1)$ for some $\lambda \in \overline{\mathbb{F}_q}^*$.  

        If $s=1$, then $G(X,Y)=g_0 Y+g_1(X)$. This means that $F_{b,c,r}(X,\widetilde{g}(X))\equiv 0$ for some $\widetilde{g}(X)\in \overline{\mathbb{F}_q}[X]$. Let $\alpha\geq 0 $ be the degree of $\widetilde{g}(X)$. From $F_{b,c,r}(X,\widetilde{g}(X))\equiv 0$ we deduce that 
        $$2^r+1+\alpha(2^r+1)=\alpha+1, $$
        a contradiction to $\alpha\geq 0$.

        Thus $s\geq 2$. Arguing as above, we consider  the coefficient of $Y^{2^r+1-\ell}$, $\ell\in [2 \ldots s]$, in  $G(X,Y)H(X,Y)$ and in $F_{b,c,r}(X,Y)$. Using induction, $h_0(X)\mid h_\ell(X)$, $\ell=2,\ldots, s$. 
        Also, the coefficient of $Y^{2^r+1-s-\ell}$, $\ell\in [1 \ldots 2^r+1-2s]$, in $G(X,Y)H(X,Y)$ is 
        $$g_0(X) h_{s+\ell}(X)+g_1(X)h_{s+\ell-1}(X)+\cdots+g_s(X)h_{\ell}(X).$$
        Since it vanishes, we obtain, again by induction, $h_0(X)\mid h_{s+\ell}(X)$, for $\ell\leq 2^r+1-2s$. Thus $h_0(X)\mid H(X,Y)$ and again we find a non-costant factor of $F_{b,c,r}$ of degree $0$ in $Y$, a contradiction.         
    \end{enumerate}
\end{proof}

Therefore, we get the following non-existence result.
\begin{theorem}\label{th:main}
    Let $m\ge 2$ be an even integer. Then, 
    \begin{enumerate}
        \item for any $b\in\mathbb{F}_{2^m}$, $c\in \mathbb{F}_{2^m}^*$ and $r< m/8-1$ the function $f_{b,c,r}$ is not APN;
        \item for any $b\in\mathbb{F}_{2^m}^*$ and $r< m/2$ the function $f_{b,0,r}$ is not APN;
    \end{enumerate}
\end{theorem}
\begin{proof}
Let $q=2^{m/2}$. From Theorem \ref{th:Cbc} we get that the curve $\mathcal{D}_{b,c,r}$, defined over $\mathbb{F}_q$, is absolutely irreducible of order $d=2^{r+1}+2$. Now, applying the Hasse-Weil bound (Theorem \ref{Th:AubryPerret}), noting that the curve has two points at the infinity, we get that the number of $\mathbb{F}_q$-rational (affine) points of  $\mathcal{D}_{b,c,r}$ are at least $q+1-(d-1)(d-2)\sqrt{q}-2$.
It is easy to see that if $r<\frac{m}{8}-1$  we have $ q+1-(d-1)(d-2)\sqrt{q}-2>0$.

The second case comes from Theorem \ref{th:b0}.
\end{proof}

\begin{remark}
For the case $r=1$, Theorem \ref{th:main} implies that for $m\ge 18$ the function $f_{b,c,1}$ cannot be APN for any choice of $b, c\in\mathbb{F}_{2^m}$. 
\end{remark}

In \cite{APN}, the authors show that for $m\le 6$, we have instance of APN functions coming from Theorem \ref{Th:APN} for $r=1$. To check that $f_{b,c,1}$ cannot be APN for $8 \le m\le 16$ we need the following proposition which allows us to reduce the number of pairs $(b,c)$.
\begin{prop}\label{prop:red}
    Let $k\ge 0$ be an integer, and $u\in\mu_{q+1}$. Then, for any $b,c\in\mathbb{F}_{q^2}$ the equation
    \begin{equation}\label{eq:sol1}
            \big(cx^{2^r+1} + bx^{2^r} + 1\big)^{q+1} + x^{q+1}=0
    \end{equation}

    admits a solution over $\mathbb{F}_{q^2}$ if and only if 
    \begin{equation}\label{eq:sol2}
    \big(c^{2^k}u^{2^k(2^r+1)}x^{2^r+1} + b^{2^k}u^{2^{k+r}}x^{2^r} + 1\big)^{q+1} + x^{q+1}=0
\end{equation}
    has a solution.
\end{prop}
\begin{proof}
If we perform the change of variable $x\mapsto ux^{2^{-k}}$, then \eqref{eq:sol1} becomes
$$
\big(cu^{(2^r+1)}x^{2^{-k}(2^r+1)} + bu^{2^{r}}x^{2^{r-k}} + 1\big)^{q+1} + x^{2^{-k}(q+1)}=0.
$$
Now, raising it to the power of $2^k$ we get \eqref{eq:sol2}.
\end{proof}

{\begin{remark}
    Proposition \ref{prop:red} permits to reduce the number of pairs $(b,c)\in\mathbb{F}_{q^2}\times \mathbb{F}_{q^2}^*$, and thus of polynomials $P_{c,b}(X)$, for checking the existence of an APN function as in Theorem \ref{Th:APN}. Indeed, let $b\in \mathbb{F}_{q^2}$, and let $B_b:=\{b^{2^k}\cdot u^{2^{r+k}}\,:\, u \in\mu_{q+1}, \, 0\le k\le m\}$. Then, Proposition \ref{prop:red} implies that if for any $c\in \mathbb{F}_{q^2}^*$ the polynomial $P_{c,b}(X)$ admits always a solution in $\mathbb{F}_{q^2}$, then for any $b'\in B_b$ we have that $P_{c,b'}(X)$ admits always a solution in $\mathbb{F}_{q^2}$ for any $c\in \mathbb{F}_{q^2}^*$.

Therefore, we can partition $\mathbb{F}_{q^2}$ in sets of type $B_b$ and restrict the analysis to one representative for each set. For example, let $m=16$, using Proposition \ref{prop:red} we reduce the analysis to $36\cdot (2^{16}-1)$ pairs instead of $2^{16}\cdot (2^{16}-1)$.
\end{remark}}


By \eqref{eq0}, the existence of a root of the polynomial  
\[
P_{c,b}(X) = \bigl(cX^{2^r+1} + bX^{2^r} + 1\bigr)^{q+1} + X^{q+1}
\]  
is equivalent to the existence of an element $u \in \mu_{q+1}$ such that the equation  
\[
cx^{2^r+1} + bx^{2^r} + ux + 1 = 0
\]  
admits a root in $\mathbb{F}_{q^2}$. This equation can be transformed into  
\begin{equation}\label{eq:projpol}
x^{2^r+1} + x + A = 0,
\end{equation}  
where  
\[
    A = \frac{(ub+c)c^{2^r-1}}{\bigl(uc^{2^r-1} + b^{2^r}\bigr)^{2^{-r}+1}},  
\]
under the assumption that $uc^{2^r-1}+b^{2^r} \neq 0$, see for instance \cite{Bluher}. In \cite[Theorem 2.1]{BTThex} it has been proved that equation \eqref{eq:projpol} admits no solution over $\mathbb{F}_{q^2}$ if and only if  
\begin{equation}\label{eq:condition}
A = \frac{a(a+1)^{2^r+2^{-r}}}{(a+a^{2^{-r}})^{2^r+1}},
\end{equation}
for some non-cube $a$.
For the case $r=1$, the previous request is equivalent to ask that 
\begin{align*}
A=a+\frac{1}{a},
\end{align*}
for some non-cube $a$.
So, for $r=1$, using MAGMA \cite{Magma} it is possible to check that one can always find some $u \in \mu_{q+1}$ such that {$uc+b^{2} \neq 0$ and} the associated value of $A$ does \emph{not} belong to the set  
\[
\Biggl\{ a+\frac{1}{a} : a \text{ not a cube} \Biggr\},
\]
for any choice of $b,c\in\mathbb{F}_{2^m}$ and $8\le m\le 16$. Therefore, the function $f_{b,c,1}$ cannot be APN. So, we get the following result.
\begin{cor}
      Let $m\ge 8$ be an even integer. Then, for any choice $b,c \in \mathbb{F}_{2^m}$ the function $f_{b,c,1}$ as in Theorem \ref{Th:APN} is not APN.
\end{cor}

Moreover, by a computer check we could verify that $f_{b,c,r}$ cannot be APN also for $8\le m\le 22$, and $r<m/2$ with $\gcd(r,m)=1$. {In particular, for any $r<m/2$ and any choice of $b,c$, there exists $u\in\mu_{q+1}$ such that $uc^{2^r-1}+b^{2^r} \neq 0$ and \eqref{eq:condition} is \emph{not} satisfied for any non-cube $a$.}
Therefore, we conjecture the following:
\begin{conj}
    Let $m\ge 8$ be an even integer. Then, for any choice $b,c \in \mathbb{F}_{2^m}$ and $r\le m/2$, $\gcd(r,m)=1$, the function $f_{b,c,r}$ as in Theorem \ref{Th:APN} is not APN.
\end{conj}

\section{Concluding remarks}

In this work, using algebraic-geometric tools we proved that the bivariate construction of APN functions introduced in \cite{APN} cannot yield APN functions whenever  
$r < \frac{m}{8} - 1$. 
In particular, for the case $r=1$, this implies that the function cannot be APN for $m \geq 18$. Moreover, by performing computations in \textsf{MAGMA} we established that for $m \geq 8$ there are no APN functions from this class when $r=1$. These results naturally lead to the following conjecture:
\[
\text{For } m \geq 8 \text{ and } r < \frac{m}{2}, \text{ no APN function arises from Theorem \ref{Th:APN}.}
\]
Computationally this conjecture holds for $8\le m\le 22$.

If true, this conjecture implies that the Calderini et al.\ construction does not generate an infinite family of APN functions, but only sporadic examples for $m \leq 6$, where APN functions of this type indeed exist, as shown in \cite{APN}.

\medskip

A possible way to investigate this conjecture could be trying to show that for any choice of parameters $b, c$, one can always find some $u \in \mu_{q+1}$ such that the associated value of $A$, as in \eqref{eq:projpol}, does \emph{not} belong to the set  
\[
\Biggl\{ \frac{a(a+1)^{2^r+2^{-r}}}{(a+a^{2^{-r}})^{2^r+1}} : a \text{ not a cube} \Biggr\}.
\]

An equivalent approach could be that of studying the permutation property of certain linearized polynomials. Indeed, projective polynomials are related to linearized polynomials. In particular, noting that $x^{2^r-1}$ permutes $\mathbb{F}_{2^{m}}$ when $\gcd(r,m)=1$, the polynomial $P_A(X)=X^{2^r+1} + X + A$  has no roots over $\mathbb{F}_{2^{m}}$ if and only if $L_A(X)=X^{2^{2r}} + X^{2^r} + AX$ is a permutation polynomial, since $P_A(X^{2^r-1})X=L_A(X)$. 

Linearized polynomials of this form and their zeros have been studied in several works (see for instance \cite{choelin,HKproj2,MGS}).

\bigskip
{\footnotesize
\noindent\textit{Acknowledgments.}
The research was partially supported by the Italian National Group for Algebraic and Geometric Structures and their Applications (GNSAGA–INdAM).\\
Part of this work was carried out during a Research in Pairs program at the Centro Internazionale per la Ricerca Matematica (CIRM), Trento. The authors gratefully acknowledge CIRM for its financial and logistical support.\\
The research of M. Calderini was partially
supported by MUR–Italy through PRIN 2022RFAZCJ "Algebraic methods in cryptanalysis"}

\end{document}